\documentclass[11pt, a4paper]{article}
\usepackage[latin1]{inputenc}
\usepackage[T1]{fontenc}
\usepackage[english]{babel}
\usepackage{amssymb}
\usepackage{amsmath}
\usepackage{amsthm}
\usepackage[colorlinks=true, allcolors=blue]{hyperref}
\usepackage{geometry}
\usepackage{fancyhdr}
\usepackage{relsize}
\usepackage{upgreek}
\usepackage{verbatim}
\usepackage{enumitem}
\usepackage{calrsfs}
\usepackage{tikz}
\usetikzlibrary{matrix,arrows,calc}
\usepackage{mypackage}
\newcommand{\rla}{\rightleftarrows}
\newcommand{\uppertri}{\big( \begin{smallmatrix} R_d & * \\ & R_e \end{smallmatrix} \big)}
\newcommand{\uppertripq}{\big( \begin{smallmatrix} R_p & * \\ & R_q \end{smallmatrix} \big)}
\newcommand{\fpo}{\big( \begin{smallmatrix} F_p \\ 0\end{smallmatrix} \big)}
\title{On the Semi-Stable CoHa and its Modules Arising from Smooth Models}
\author{H. Franzen\footnote{\myaddress}}
\date{ }
\begin{document}
	\maketitle
	\begin{abstract}
		We study a variant of the semi-stable Cohomological Hall algebra which we construct using equivariant Chow groups. This algebra, we call it the semi-stable ChowHa, arises as a quotient of the CoHa. Smooth models of quiver moduli give rise to modules over the semi-stable ChowHa. We prove that these modules are cyclic and we compute a presentation using Harder--Narasimhan methods.
	\end{abstract}

	\section*{Introduction\markboth{}{Introduction}}

	The Cohomological Hall algebra---which we will call CoHa, for brevity---and its semi-stable version were invented by Kontsevich and Soibelman in \cite{KS:11}. These algebras have proven to be very useful tools in the theory of Donaldson--Thomas invariants. For example, a result of Efimov (cf.\ \cite{Efimov:12}) shows that the CoHa of a symmetric quiver is a free super-commutative algebra. This implies that the Donaldson--Thomas invariants are non-negative and integral.

	There is a class of modules over the semi-stable CoHa of a quiver arising from smooth models (a.k.a. framed stable quiver moduli) of the underlying quiver moduli space. These modules were introduced by Soibelman (cf.\ \cite[Section 4]{Soibelman:14}). This comprises an important special case: When choosing the trivial stability condition, the smooth models---which in this case are usually called non-commutative Hilbert schemes---have a cell decomposition which can be described entirely combinatorially (cf.\ \cite[Theorem 1.3]{Reineke:05} and \cite[Corollary 7.8]{ER:09}). In \cite[Section 2]{Reineke:12}, Reineke relates the generating series of the CoHa of the $m$-loop quiver with the Poincar\'{e} series of non-commutative Hilbert schemes (over the path algebra of the $m$-loop quiver, which is nothing but a free non-commutative algebra in $m$ letters). Therefore, we want to examine the interrelation of the Cohomological Hall algebra and the modules over it obtained by non-commutative Hilbert schemes, or, 
	more generally, the connection between the semi-stable CoHa and the modules coming from smooth models.

	When restricting to the case of a trivial potential, we can define an analog of the semi-stable CoHa using Edidin--Graham's equivariant Chow groups. Let's call it the semi-stable ChowHa. The ``ordinary'' ChowHa---which coincides with the semi-stable ChowHa for the trivial stability---is isomorphic to the CoHa. The semi-stable ChowHa is a quotient of the ChowHa/CoHa and under the hypothesis that the closure of every Harder--Narasimhan (HN) stratum is a union of HN strata, we are able to compute the kernel explicitly in terms of the ChowHa-multiplication (Lemma \ref{kack_lemma}). Although not applicable in most cases, we are able to use this lemma to prove Proposition \ref{tensorProd} which states that the CoHa of the quiver of type $\tilde{A}_1$ with the symmetric orientation is isomorphic to the (descending) tensor product of the semi-stable CoHa's---with respect to a fixed non-trivial stability; in this case the semi-stable CoHa's and the semi-stable ChowHa's agree. Such an isomorphism is known to exist for the CoHa of the quiver $A_2$ thanks to Kontsevich--Soibelman (cf.\ \cite[5.2]{KS:11}) or, more generally, for the CoHa of any Dynkin quiver except for type $E_8$ thanks to Rim{\'a}nyi (cf.\ \cite{Rimanyi:13}). 

	We are considering modules over the semi-stable ChowHa which are obtained by Chow groups of smooth models. There is a natural map from the semi-stable ChowHa---regarded as a left-module over itself---to the aforementioned module which is surjective and whose kernel can be described explicitly in terms of equivariant Chern classes of universal bundles (Theorem \ref{main_thm}). As the closure of a framed HN stratum which lies over the unframed semi-stable locus is a union of framed HN strata, the proof can be completed with the same methods as the proof of Lemma \ref{kack_lemma}. Theorem \ref{main_thm} can be viewed as a generalization of \cite[Theorem 3.6]{Franzen:13:NCHilb_Loop}; it describes the case of the $m$-loop quiver.

	The paper is organized as follows: In the first section, we recollect Kontsevich--Soibelman's definition and some facts about the CoHa of a quiver (with trivial potential). We do the same in Section 2 for the semi-stable CoHa (with respect to a stability condition in the sense of King \cite{King:94}). Describing the Hecke correspondences which induce the multiplication, it follows that the construction of the semi-stable CoHa may also be carried through for equivariant Chow groups yielding the semi-stable ChowHa. Viewed as a quotient of the ChowHa/CoHa, we give a presentation under some strong assumptions about the HN stratification (Lemma \ref{kack_lemma}). The third section deals with CoHa-modules. We give the definition of a smooth model and describe the Hecke correspondences which induce the module structure. Again, this construction works in both the cohomological and the intersection theoretic setup. Due to the nice geometric structure of the smooth models, the thus induced modules coincide as graded 
	abelian groups (if the quiver has no oriented cycles). Theorem \ref{main_thm} gives a close connection between the semi-stable ChowHa and the module. In Section 4, we apply this result to the special case of non-commutative Hilbert schemes.

	\begin{ack*}
		I am grateful to Markus Reineke for introducing me to the subject of Cohomological Hall algebras and for several very inspiring discussions concerning the methods and results of this work. I would also like to thank Ben Davison, Sergey Mozgovoy, Richard Rim{\'a}nyi, Yan Soibelman, and Matt Young for very helpful remarks and discussions. While doing this research, I was supported by the DFG SFB / Transregio 45 ``Perioden, Modulr\"aume und Arithmetik algebraischer Variet\"aten''.
	\end{ack*}
	
	\section{Cohomological Hall Algebra}

	A \textbf{quiver} is a finite oriented graph. We denote its set of vertexes by $Q_0$ and its set of arrows by $Q_1$. Let $Q$ be a quiver which we consider fixed, whence we often suppress it in the notations. For a dimension vector $d$, let $R_d$ be the vector space
	$$
		R_d(Q) = \bigoplus_{\alpha: i \to j} \Hom(\C^{d_i},\C^{d_j})
	$$
	on which the linear algebraic group $G_d := \prod_{i \in Q_0} \Gl_{d_i}$ acts via base change.

	\subsection*{Construction of the CoHa}

	For a dimension vector $d$, define $\HH_d$ to be the singular $G_d$-equivariant cohomology group with rational coefficients
	$$
		\HH_d(Q) = H_{G_d}^*(R_d;\Q).
	$$
	Although not always necessary, we will only use cohomology/Chow groups with rational coefficients. Therefore, most of the time, we will not indicate this in the notation. Abbreviate $\Gamma := \Z_{\geq 0}^{Q_0}$. The $\Gamma$-graded abelian group 
	$$
		\HH(Q) := \bigoplus_{d} \HH_d(Q)
	$$ 
	can be equipped with a multiplication: For dimension vectors $d$ and $e$, we consider the vector space $\uppertri$ of those $M \in R_{d+e}$ such that $M_{\alpha}(\C^{d_i}) \sub \C^{d_j}$ for all $\alpha: i \to j$; here, $\C^{d_i} \sub \C^{d_i+e_i}$ is the subspace spanned by the first $d_i$ coordinate vectors. There are maps
	$$
		R_d \times R_e \ot \uppertri \to R_{d+e},
	$$
	the right-hand map embedding $\smash{\uppertri}$ as a linear subspace into $R_{d+e}$ and the map $\smash{\uppertri} \to R_d \times R_e$ assigning to a representation $M = \left( \begin{smallmatrix} M' & * \\ & M'' \end{smallmatrix} \right)$ the pair $(M',M'')$. We have actions of the groups $L = G_d \times G_e$ on $R_d \times R_e$, of $P := \smash{\big( \begin{smallmatrix} G_d & * \\ & G_e \end{smallmatrix} \big)}$ on the space $\smash{\uppertri}$, and $G_{d+e}$ on $R_{d+e}$. With respect to these actions, $\smash{\uppertri} \to R_d \times R_e$ is an $L$-equivariant vector bundle and $\smash{\uppertri} \to R_{d+e}$ is the zero section of the $P$-equivariant vector bundle $R_{d+e} \to \uppertri$ (which forgets the south-western blocks of the matrices). We thus obtain an isomorphism and a push-forward map
	$$
		H^k_L(R_d \times R_e) \xto{}{\cong} H^k_L\uppertri \quad \text{and} \quad H^k_P\uppertri \to H^{k+s_1}_P(R_{d+e}).
	$$
	Here, $s_1 = 2\dim_{\C} R_{d+e} - 2\dim_{\C} \smash{\uppertri} = 2 \sum_{\alpha: i \to j} d_ie_j$. As $L$ is the Levi subgroup of $P$, the quotient $P/L$ is an affine space, whence the restriction map $H_P^k\smash{\uppertri} \to H_L^k\smash{\uppertri}$ is an isomorphism. Moreover, as $P$ is a parabolic subgroup of $G = G_{d+e}$, the quotient $G/P$ is projective and thus, there exists a map $H_P^k(R_{d+e}) \to H_G^{k+s_0}(R_{d+e})$. Here, $s_0$ equals $-2 \dim_{\C} G/P = -2 \sum_i d_ie_i$. Composing these maps, we obtain for integers $k,l$
	\begin{center}
		\begin{tikzpicture}[description/.style={fill=white,inner sep=2pt}]
			\matrix(m)[matrix of math nodes, row sep=1.5em, column sep=1.5em, text height=1.5ex, text depth=0.25ex]
			{
				H_{G_d}^k(R_d) \otimes H_{G_e}^l(R_e) && H_G^{k+l+s_1+s_0}(R_{d+e}) \\
				H^{k+l}_L(R_d \times R_e) & H_L^{k+l}\uppertri \cong H_P^{k+l}\uppertri & H_P^{k+l+s_1}(R_{d+e}). \\
			};
			\path[->, font=\scriptsize]
			(m-1-1) edge 	(m-1-3)
			(m-1-1) edge node[auto] {$\times$}	(m-2-1)
			(m-2-1) edge node[auto] {$\cong$}	(m-2-2)
			(m-2-2) edge 	(m-2-3)
			(m-2-3) edge	(m-1-3);
		\end{tikzpicture}
	\end{center}
	The Euler form $\chi = \chi_Q$ of $Q$ is the bilinear form on $\Z^{Q_0}$ defined by $\chi(d,e) = \sum_{i \in Q_0} d_ie_i - \sum_{\alpha:i \to j} d_ie_j$. Note that $s_1 + s_0 = -2 \sum_i d_ie_i + 2 \sum_{\alpha: i \to j} d_ie_j$ which is precisely $-2 \chi( d,e )$. We have constructed a $\Q$-linear map
	$$
		\HH_d \otimes \HH_e \to \HH_{d+e}
	$$
	which we denote $*$. Kontsevich--Soibelman show that $\HH$ thus becomes an associative, $\Gamma$-graded algebra. 

	\begin{defn*}
		The $\Gamma$-graded algebra $\HH(Q)$ is called the \textbf{Cohomological Hall algebra} of $Q$.
	\end{defn*}

	In \cite{KS:11}, the multiplication is computed explicitly: The equivariant cohomology group $H^*_{G_d}(R_d) \cong H^*_{G_d}(\pt)$ is isomorphic to
	$$
		\Q[x_{i,\nu} \mid i \in Q_0,\ 1 \leq \nu \leq d_i]^{W_d}
	$$
	with $W_d = \prod_i S_{d_i}$. The variables $x_{i,1},\ldots,x_{i,d_i}$ (living in degree $2$) may be interpreted as the Chern roots of the $G_d$-linear vector bundle $R_d \times \C^{d_i} \to R_d$ with $G_d$ acting by its $i$\textsuperscript{th} factor. Using this identification, the multiplication of the CoHa is given as follows:

	\begin{thm}[{\cite[Theorem 2]{KS:11}}] \label{Thm_KS}
		For $f \in \HH_d$ and $g \in \HH_e$, the product $f * g$ equals the function
		$$
			\sum \sigma.\left(f(\mathbf{x'})\cdot g(\mathbf{x}'') \cdot \frac{\Delta_1(\mathbf{x})}{\Delta_0(\mathbf{x})} \right).
		$$
		The above sum ranges over all $(d,e)$-shuffles $\sigma$. These are elements $\sigma = (\sigma_i \mid i) \in W_{d+e}$ such that every $\sigma_i$ is a $(d_i,e_i)$-shuffle permutation. The symbols $\mathbf{x}'$, $\mathbf{x}''$ and $\mathbf{x}$ stand for the sets of variables
		\begin{align*}
			\mathbf{x}'  &= \{ x_{i,\nu} \mid i \in Q_0,\ \nu = 1,\ldots,d_i \} \\
			\mathbf{x}'' &= \{ x_{i,\nu} \mid i \in Q_0,\ \nu = d_i+1,\ldots,d_i+e_i \} \\
			\mathbf{x}   &= \{ x_{i,\nu} \mid i \in Q_0,\ \nu = 1,\ldots,d_i+e_i \}.
		\end{align*}
		Moreover, the polynomials $\Delta_0$ and $\Delta_1$ are defined by
		\begin{align*}
			\Delta_1(\mathbf{x}) &= \prod\limits_{\alpha: i \to j}\prod\limits_{\mu=1}^{d_i}\prod\limits_{\nu=d_j+1}^{d_j + e_j} (x_{j,\nu} - x_{i,\mu}) \\
			\Delta_0(\mathbf{x}) &= \prod\limits_{i \in Q_0}\prod\limits_{\mu = 1}^{d_i} \prod\limits_{\nu=d_i+1}^{d_i+e_i} (x_{i,\nu} - x_{i,\mu} ).
		\end{align*}
	\end{thm}

	\subsection*{Construction with Equivariant Chow Groups}

	We can also define the CoHa using Edidin--Graham's equivariant intersection theory (cf.\ \cite{EG:98}). For every $d \in \Gamma$, we put
	$$
		\AA_d(Q) = A_{G_d}^*(R_d)_{\Q}.
	$$
	We define $\AA(Q)$ to be the direct sum over all these abelian groups $\AA_d$. We know that the equivariant cycle map $A_G^*(V) \to H_G^{*}(V)$ is an isomorphism of graded rings (which doubles degrees) if $V$ is a vector space and $G$ is a reductive group (or a parabolic of a reductive group) which acts linearly on $V$. So, as $\Gamma$-graded abelian groups, $\HH$ and $\AA$ coincide. Moreover, the Hecke correspondences
	$$
		R_d \times R_e \ot \uppertri \to R_d
	$$
	described above give suitable maps in equivariant intersection theory which make $\AA$ into a $\Gamma$-graded algebra. The cycle maps are compatible with these maps, whence the isomorphism $\AA \to \HH$ of graded abelian groups is actually an isomorphism of $\Gamma$-graded algebras. If we want to stress that we are working with the intersection theoretic version of the CoHa, we will call it the \textbf{ChowHa}.

	\subsection*{The Symmetric Case}

	If $Q$ is a \emph{symmetric} quiver, i.e.\ for all vertexes $i,j$, there are as many arrows from $i$ to $j$ as from $j$ to $i$, then the Euler form is a symmetric bilinear form. We can then define a refined grading on the CoHa which makes it a $(\Gamma \times \Z)$-graded algebra (cf.\ \cite[Section 2.6]{KS:11}): We put
	$$
		{\HH}_{(d,i)} = H_{G_d}^{i-\chi(d,d)}(R_d)
	$$
	and see that the CoHa-multiplication $*$ maps ${\HH}_{(d,i)} \otimes {\HH}_{(e,j)} \to {\HH}_{(d+e,i+j)}$. The thus obtained $(\Gamma \times \Z)$-graded algebra ${\HH} = \bigoplus_{d,i} {\HH}_{(d,i)}$ can be made into a graded super-commutative algebra: We define the parity of an element of ${\HH}_{(d,i)}$ to be $\epsilon(d) := \chi( d,d )\ (\mathrm{mod}\ 2)$. Using Theorem \ref{Thm_KS}, we can see that for $f \in {\HH}_{(d,i)}$ and $g \in {\HH}_{(e,j)}$, we have
	$$
		f * g = (-1)^{\chi( d,e )} g * f
	$$
	but this does not mean that the multiplication $*$ is super-commutative. Instead, it is possible to twist this multiplication with an appropriate sign making it super-commutative. There exists (see \cite[Section 2.6]{KS:11}) a bilinear form $\psi: (\Z/2\Z)^{Q_0} \times (\Z/2\Z)^{Q_0} \to (\Z/2\Z)$ such that
	$$
			\psi(d,e) + \psi(e,d) = \chi( d,e ) + \epsilon(d)\epsilon(e)\ (\mathrm{mod}\ 2).
	$$
	Thus, the twisted product $f \star g := (-1)^{\psi(d,e)} f*g$ makes ${\HH}$ into a graded super-commutative algebra. Of course, the same construction applies for the ChowHa as well (since the ChowHa is isomorphic to the CoHa). Let $P_Q(q,t) = \sum_d \sum_k (-1)^k \dim\! \big( \HH_{(d,k)}(Q) \big) q^{k/2} t^d \in \Q(q^{1/2})[[t_i \mid i]]$ be the generating series of the CoHa of a symmetric quiver $Q$. Using that the generating series of the ring of symmetric polynomials in $n$ variables is $(1 - q)^{-1} \ldots (1 - q^n)^{-1}$, we see that
	$$
		P_Q(q,t) = \sum_d (-q^{1/2})^{\chi( d,d )} \prod_i \prod_{\nu=1}^{d_i} (1-q^\nu)^{-1} t^d.
	$$
	In \cite[Corollary 3]{KS:11}, it is shown that the generating series of the CoHa has a product expansion
	$$
		P_Q(q,t) = \prod_d \prod_k \prod_{n \geq 0} (1 - q^{n+k/2}t^d)^{(-1)^{k-1}c_{(d,k)}},
	$$
	Observing that the generating series of a free super-commutative algebra generated by one element in bidegree $(d,k)$ is $(1 - q^{k/2}t^d)^{(-1)^{k-1}}$, this product expansion led Kontsevich--Soibelman to a conjecture which was proved by Efimov:

	\begin{thm}[{\cite[Theorem 1.1]{Efimov:12}}] \label{efimov}
		For a symmetric quiver $Q$, the algebra $\HH(Q)$ is isomorphic to a free super-commutative algebra over a $(\Gamma \times \Z)$-graded vector space
		$$
			V = V^{\mathrm{prim}} \otimes \Q[z],
		$$
		where $z$ has bidegree $(0,2)$ and such that $\bigoplus_k V_{(d,k)}^{\mathrm{prim}}$ is finite-dimensional for every $d$.
	\end{thm}

	From Efimov's theorem, it follows that the numbers $c_{(d,k)}$ in the product expansion of $P_Q(q,t)$ are positive integers, namely $c_{(d,k)} = \smash{\dim V_{(d,k)}^{\mathrm{prim}}}$.

	\begin{conv*}
		If the quiver is symmetric, we will always understand its CoHa/ChowHa to be equipped with the $(\Gamma \times \Z)$-grading and the super-commutative multiplication, unless otherwise stated.
	\end{conv*}

	\subsection*{Examples}

	We describe the structure of the CoHa in three examples that will accompany us throughout the article. For the explicit descriptions, it is necessary to work with rational coefficients. Examples \ref{exA} and \ref{exB} can also be found in \cite[Section 2.5]{KS:11}.

	\begin{ex} \label{exA}
		Let $\ptquiv$ be the quiver consisting of a single vertex and no arrows. In this case, $\HH(\ptquiv) = \bigoplus_{d \geq 0} \HH_d$ with
		$$
			\HH_d = \Q[x_1,\ldots,x_d]^{S_d}.
		$$
		The Euler form of $\ptquiv$ is $\chi( d,e ) = de$. Therefore, a symmetric polynomial $f \in \HH_d$ homogeneous of degree $n$ lives in bidegree $(d,2n+d^2)$ with respect to the refined grading from the previous section. The product of $f \in \HH_d$ and $g \in \HH_e$ is given by
		$$
			\sum f(x_{\sigma_1},\ldots,x_{\sigma_d})\cdot g(x_{\sigma_{d+1}},\ldots,x_{\sigma_{d+e}}) \cdot \frac{1}{\prod_{\mu=1}^d \prod_{\nu=d+1}^{d+e} (x_{\sigma_\nu}-x_{\sigma_{\mu}})}.
		$$
		We observe that in this case, the usual product is already super-commutative. We will identify ${\HH}$ with the exterior algebra over ${\HH}_1$. We see that $f * f = 0$ for every $f \in {\HH}_1$, whence we obtain a natural homomorphism $\bigwedge({\HH}_1) \to {\HH}$ of $(\Z_{\geq 0} \times \Z)$-graded algebras. Consider the elements $\psi_i \in {\HH}_1$ defined by $\psi_i(x) = x^i$ (the power is taken with respect to the usual multiplication of polynomials, \emph{not} the CoHa-multiplication). Then, $\psi_i \in {\HH}_{(1,2i+1)}$ and $\psi_0,\psi_1,\ldots$ form a basis of $\HH_1$. An induction shows that for $0 \leq k_1 < \ldots < k_d$, we have
		$$
			(\psi_{k_1} * \ldots * \psi_{k_d})(x_1,\ldots,x_d) = s_\lambda(x_1,\ldots,x_d),
		$$
		where $s_\lambda$ is the Schur function belonging to the partition $\lambda = (k_d-d+1,\ldots,k_2-1,k_1)$. Hence, the induced homomorphism $\bigwedge(\psi_0,\psi_1,\ldots) \to {\HH}$ is surjective. A comparison of the generating series shows that it is in fact an isomorphism. So, in this case, 
		$$
			V^{\mathrm{prim}} = {\HH}_{(1,1)} = \Q \cdot \psi_0,
		$$
		the one-dimensional bigraded vector space concentrated in bidegree $(1,1)$.
	\end{ex}

	\begin{ex} \label{exB}
		Let $\loopquiv$ be the quiver with one vertex and one loop. As a $\Z_{\geq 0}$-graded vector space, the CoHa $\HH(\loopquiv)$ coincides with the CoHa of $\ptquiv$. However, the multiplication differs. For $f \in \HH_d$ and $g \in \HH_e$, the product $f * g$ equals
		$$
			\sum f(x_{\sigma_1},\ldots,x_{\sigma_d})\cdot g(x_{\sigma_{d+1}},\ldots,x_{\sigma_{d+e}}).
		$$
		The Euler form of the loop quiver is trivial, wherefore an $f \in \HH_d$ which is homogeneous of degree $n$ is located in $\HH_{(d,2n)}$. Similar to Example \ref{exA}, we consider the natural map $\Sym(\HH_1) \to \HH$ from the symmetric algebra over $\HH_1$ to the CoHa. Let $\phi_i(x) = x^i \in \HH_1$. It lives in bidegree $(1,2i)$. For a sequence $k_1 \geq k_2 \geq \ldots \geq k_d$, we get
		$$
			(\phi_{k_1} * \ldots * \phi_{k_d})(x_1,\ldots,x_d) = c_\lambda \cdot m_\lambda(x_1,\ldots,x_d),
		$$
		where $m_\lambda$ is the monomial symmetric function attached to the partition $\lambda = (k_1,\ldots,k_d)$ and $c_\lambda$ is some positive integer. Comparing the two generating series yields that the natural map $\Q[\phi_0,\phi_1,\ldots] \to \HH$ is an isomorphism. That means that for the loop quiver,
		$$
			V^{\mathrm{prim}} = {\HH}_{(1,0)} = \Q \cdot \phi_0
		$$
		which lives in bidegree $(1,0)$.
	\end{ex}

	\enlargethispage{\baselineskip}
		
	\begin{ex} \label{exC_first}
		Our last example is the quiver $\tildeAOne$ of type $\tilde{A}_1$ with the symmetric orientation. A dimension vector for $\tildeAOne$ is a pair $(m,n)$ of integers which we denote $m \rla n$, in order to avoid ambiguous notation. Here, $\HH(\tildeAOne) = \bigoplus_{m,n \geq 0} \HH_{m \rla n}$ with
		$$
			\HH_{m \rla n} = \Q[x_1,\ldots,x_m,y_1,\ldots,y_n]^{S_m \times S_n}.
		$$
		Let $f \in \HH_{m \rla n}$ and $g \in \HH_{r \rla s}$. The product $f * g$, which lives in $\HH_{(m+r) \rla (n+s)}$, is the polynomial
		$$
			\sum f(x_{\sigma_1},\ldots,x_{\sigma_m},y_{\tau_1},\ldots,y_{\tau_n}) g(x_{\sigma_{m+1}},\ldots,x_{\sigma_{m+r}},y_{\tau_{n+1}},\ldots,y_{\tau_{n+s}}) \cdot \frac{\Delta_1(x_{\sigma_i},y_{\tau_j})}{\Delta_0(x_{\sigma_i},y_{\tau_j})},
		$$
		where $\sigma$ is an $(m,r)$-shuffle, $\tau$ an $(n,s)$-shuffle, and
		\begin{align*}
			\Delta_1(x_1,\ldots,x_{m+r},y_1,\ldots,y_{n+s}) &= \prod_{i=1}^m \prod_{j'=n+1}^{n+s} (y_{j'} - x_{i}) \prod_{i'=m+1}^{m+r} \prod_{j=1}^n (x_{i'} - y_{j}) \\
			\Delta_0(x_1,\ldots,x_{m+r},y_1,\ldots,y_{n+s}) &= \prod_{i=1}^m \prod_{i'=m+1}^{m+r} (x_{i'} - x_{i}) \prod_{j=1}^n \prod_{j'=n+1}^{n+s} (y_{j'} - y_{j}).
		\end{align*}
		For two dimension vectors $m \rla n$ and $r \rla s$, the Euler form is given by $\chi( m \rla n, r \rla s ) = mr + ns - ms - nr$ and thus $\chi( m \rla n, m \rla n ) = (m-n)^2$. We obtain that a polynomial $f \in  \HH_{m \rla n}$ homogeneous of degree $k$ lies in bidegree $(m \rla n,2k + (m-n)^2)$.  Observing that $\chi( m \rla n, r \rla s )$ has always the same parity as $\epsilon(m \rla n)\epsilon(r \rla s)$, we see that the CoHa-multiplication $*$ is in this case already super-commutative. We will see in the following that we can construct the vector space $V^{\mathrm{prim}}$ explicitly, like in Examples \ref{exA} and \ref{exB}.
	\end{ex}
	
	\section{Semi-Stable CoHa and Semi-Stable ChowHa}

	Let's briefly recall the notion of semi-stability. In addition to fixing $Q$, we fix a stability condition $\theta$, i.e.\ a $\Z$-linear map $\Z^{Q_0} \to \Z$ (or, more generally, a $\Q$-linear map $\Q^{Q_0} \to \Q$). Whenever it is convenient, we will suppress the dependency on $Q$ and $\theta$ in the notation. Define the associated slope function $\slope = \slope_{\theta}$ by assigning to $0 \neq d \in \Gamma$ the value
	$$
		\slope(d) = \frac{\theta(d)}{\sum_i d_i}.
	$$
	Abbreviate $\slope(M) = \slope(\dimvect M)$ for any representation $M \neq 0$ of $Q$. A representation $M$ of $Q$ is called $\theta$-\textbf{semi-stable} if $\slope(M') \leq \slope(M)$ for every (non-zero) subrepresentation $M'$ of $M$. It is called $\theta$-\textbf{stable} if the above inequality is strict, unless $M' = M$. Interpreting $M$ as a point of the variety $R_d$, King has shown in \cite{King:94} that this notion of (semi-)stability can be realized as a notion of (semi\nobreakdash-)stability in the sense of Mumford's geometric invariant theory (cf.\ \cite{GIT:94}). We define $R_d^{\theta-\sst}$ to be the open subset of all $\theta$-semi-stable points of $R_d$. An easy observation shows that
	$$
		\slope(d) \leq \slope(d+e) \gdw \slope(d+e) \leq \slope(e) \gdw \slope(d) \leq \slope(e)
	$$
	for all dimension vectors $d$ and $e$ of $Q$. Therefore, given a short exact sequence $0 \to M' \to M \to M'' \to 0$ of representations of $Q$ and provided that their slopes are equal, $M$ is semi-stable if and only if both $M'$ and $M''$ are.

	\subsection*{Semi-Stable CoHa}

	The above considerations enable us to restrict the Hecke correspondences from Section 1 to the semi-stable loci (as introduced in \cite{KS:11}). Given dimension vectors $d$ and $e$ of the same slope, say $\mu$, the sum $d+e$ has also slope $\mu$ and the map $\uppertri \to R_d \times R_e$ from above restricts to a map 
	$$
		\uppertri^{\sst} := \uppertri \cap R_{d+e}^{\sst} \to R_d^{\sst} \times R_e^{\sst}.
	$$
	This map is $L$-equivariant and a vector bundle as
	\begin{center}
		\begin{tikzpicture}[description/.style={fill=white,inner sep=2pt}]
			\matrix(m)[matrix of math nodes, row sep=1.5em, column sep=3em, text height=1.5ex, text depth=0.25ex]
			{
				\uppertri^{\sst}	& R_d^{\sst} \times R_e^{\sst} \\
				\uppertri		& R_d \times R_e \\
			};
			\path[->, font=\scriptsize]
			(m-1-1) edge 	(m-1-2)
			(m-1-1) edge 	(m-2-1)
			(m-1-2) edge 	(m-2-2)
			(m-2-1) edge 	(m-2-2);
		\end{tikzpicture}
	\end{center}
	is a cartesian diagram. We may thus carry through the same construction as for the CoHa-multi\-pli\-cation and obtain a linear map
	$$
		H_{G_d}^i(R_d^{\sst}) \otimes H_{G_e}^j(R_e^{\sst}) \to H_{G_{d+e}}^{i+j-2\chi( d,e )}(R_{d+e}^{\sst}).
	$$
	Thus, when defining $\Gamma_{\mu}$ to be the submonoid of $\Gamma$ consisting of $0$ and of all $d \neq 0$ with $\slope(d) = \mu$ and putting $\HH_d^{\theta-\sst}(Q) = H_{G_d}^*(R_d^{\sst})$, we obtain a $\Gamma_\mu$-graded algebra by
	$$
		\HH^{\theta-\sst,\mu}(Q) := \bigoplus_{d \in \Gamma_\mu} \HH_d^{\theta-\sst}(Q).
	$$

	\begin{defn*}
		We call $\HH^{\theta-\sst,\mu}(Q)$ the $\theta$-\textbf{semi-stable CoHa of} $Q$ \textbf{of slope} $\mu$.
	\end{defn*}

	It is evident that, when choosing $\theta = 0$, we recover the CoHa as $\HH^{0-\sst,0}$.
	By the above cartesian diagram, we can see that pulling back along the open embeddings $R_d^{\sst} \to R_d$ yields a homomorphism of ($\Gamma_\mu$-graded) algebras
	$
		\HH^\mu \to \HH^{\sst,\mu}
	$
	where $\HH^\mu$ is the subalgebra $\bigoplus_{d \in \Gamma_\mu} \HH_d$ of $\HH$.

	\subsection*{Semi-Stable ChowHa}

	Just like in the case of the ChowHa $\AA$, we can also define a variant of the semi-stable CoHa using equivariant Chow groups, say $\AA^{\theta-\sst,\mu}(Q)$, by defining it to be the direct sum over all
	$$
		\AA_d^{\theta-\sst}(Q) = A_{G_d}^*(R_d^{\sst})
	$$
	for $d \in \Gamma_{\mu}$. Again, $\AA^{\sst,\mu}$ is a $\Gamma_\mu$-graded algebra. 

	\begin{defn*}
		We call the algebra $\AA^{\theta-\sst,\mu}(Q)$ the $\theta$-\textbf{semi-stable ChowHa of} $Q$ \textbf{of slope} $\mu$.
	\end{defn*}

	There is a homomorphism of (graded) algebras $\AA^{\sst,\mu} \to \HH^{\sst,\mu}$ which is induced by the equivariant cycle map. 

	In \cite[Proposition 2.5]{Reineke:03}, Reineke proves that every representation $M$ of $Q$ possesses a unique filtration $M = M^r \supseteq \ldots \supseteq M^1 \supseteq M^0 = 0$ such that every subquotient $M^\nu/M^{\nu-1}$ is semi-stable and which satisfies $\slope(M^1/M^0) > \ldots > \slope(M^r/M^{r-1})$. This filtration is called the \textbf{Harder--Narasimhan (HN) filtration} of $M$ (with respect to $\theta$). Denoting the dimension vector of $M^{\nu}/M^{\nu-1}$ by $d^{\nu}$, the tuple $d^* = (d^1,\ldots,d^r)$ is called the HN type of $M$. The set of HN types of $d$ of length $r$ will be denoted $\HN_r(d)$. The set $R_{d^*}^{\mathrm{HN}}$ of all representations of $Q$ having HN type $d^*$ is an irreducible, locally closed subset of $R_d$. Clearly, $R_d$ equals the disjoint union $\bigsqcup_{d^*} R_{d^*}^{\mathrm{HN}}$ ranging over all possible HN types which sum to $d$.

	For the semi-stable ChowHa, the natural map $\AA^\mu \to \AA^{\sst,\mu}$ is clearly surjective. Under some strong hypotheses about the HN stratification, we are able to give a description of the kernel of $\AA_d \to \AA_d^{\sst}$.

	\label{HN}

	\begin{lem} \label{kack_lemma}
		Let $d$ be a dimension vector such that for every HN type $d^*$ of $d$, the closure of the HN stratum $R_{d^*}^{\HN}$ is a union of HN strata. Then, the kernel of $\AA_d \to \AA_d^{\theta-\sst}$ equals
		$$
			\sum_{r \geq 2}\ \sum_{d^* \in \HN_r(d)} \AA_{d^1} * \ldots * \AA_{d^r}.
		$$
	\end{lem}

	\begin{proof}
		For a HN type $d^* = (d^1,\ldots,d^r)$ of $d$, denote by $R_{d^*}$ the closure of $R_{d^*}^{\HN}$ in $R_d$. It coincides with the subset of those $M \in R_d$ which possess a filtration
		$$
			M = M^r \supseteq \ldots \supseteq M^1 \supseteq M^0 = 0
		$$ 
		such that $\dimvect M^\nu/M^{\nu-1} = d^\nu$. Under the assumptions of the lemma, we can define an order on the set of HN types of $d$ as follows: For two HN types $d^*$ and $e^*$ of $d$, define $d^* \unrhd e^*$ if $R_{d^*}$ is contained in $R_{e^*}$. Let
		$$
			P_{d^*} = \begin{pmatrix} \Gl_{d^1} & \ldots & * \\ & \ddots & \vdots \\ & & \Gl_{d^r} \end{pmatrix}, \quad W_{d^*} = \begin{pmatrix} R_{d^1} & \ldots & * \\ & \ddots & \vdots \\ & & R_{d^r} \end{pmatrix},\ \text{and} \quad %
			W_{d^*}^o = \begin{pmatrix} R_{d^1}^{\sst} & \ldots & * \\ & \ddots & \vdots \\ & & R_{d^r}^{\sst} \end{pmatrix}.
		$$
		The group $P_{d^*}$ is a parabolic subgroup of $G = G_d$ whose Levi factor is $L_{d^*} = G_{d^1} \times \ldots \times G_{d^r}$. The projection $W_{d^*} \to R_{d^1} \times \ldots \times R_{d^r}$ is an $L_{d^*}$-equivariant vector bundle whence we have isomorphisms
		$$
			A_{L_{d^*}}^{j+\chi( d^1,\ldots,d^r)}(R_{d^1} \times \ldots \times R_{d^r}) \cong A_{P_{d^*}}^{j+\chi( d^1,\ldots,d^r)}(W_{d^*}) \cong A^G_n(W_{d^*} \times^{P_{d^*}} G)
		$$
		with $n = \dim R_d - j$ and $\chi( d^1,\ldots,d^r ) = \smash{\sum_{j < l}} \chi( d^j,d^l )$. We note that the equivariant product map $A_{G_d}^*(R_d^{\sst}) \otimes A_{G_d}^*(R_e^{\sst}) \to A_{G_d \times G_e}^*(R_d^{\sst} \times R_e^{\sst})$ is surjective (even with integral coefficients) as Totaro's argument from \cite[Lemma 6.1]{Totaro:99} can also be applied to equivariant Chow rings. Therefore, we are bound to show that the sequence
		\begin{equation} \tag{*}
			\bigoplus_{d^* \rhd (d)} A_n^G(W_{d^*} \times^{P_{d^*}} G) \xto{}{\psi} A_n^G(R_d) \to A_n^G(R_d^{\sst}) \to 0
		\end{equation}
		is exact. The map $\psi$ is given by the sum of the push-forwards of $W_{d^*} \times^{P_{d^*}} G \to R_d$. By the well known exact sequence for (equivariant) Chow groups, we have an exact sequence $A_n^G(R_d^{\unst}) \to A_n^G(R_d) \to A_n^G(R_d^{\sst}) \to 0$, whence it suffices to show that $\psi$ induces a surjection onto $A_n^G(R_d^{\unst})$. For every HN type $d^*$, let $R_{d^*}^c$ be the complement of $R_{d^*}^{\HN}$ in $R_{d^*}$. In particular, choosing $d^* = (d)$, the unstable locus coincides with $R_{(d)}^c$. We show by induction on $d^*$ that the map
		$$
			\psi_{e^*}: \bigoplus_{d^* \rhd e^*} A_n^G(W_{d^*} \times^{P_{d^*}} G) \to A_n^G(R_{e^*}^c)
		$$
		sending $\alpha_{d^*}$ to $\psi_{d^*,e^*}(\alpha_{d^*})$ is surjective---which then completes the proof of the lemma. Here, $\psi_{d^*,e^*}$ is the push-forward of the proper morphism $W_{d^*} \times^{P_{d^*}} G \to R_{d^*} \to R_{e^*}^c$. For $e^*$ maximal, the desired surjectivity is obvious, whence we proceed to a non-maximal $e^*$. Let $\gamma_{e^*} \in A_n^G(R_{e^*})$. By assumption, $R_{e^*}^c = \bigcup_{d^* \rhd e^*} R_{d^*}$ whence the map $\bigoplus_{d^* \rhd e^*} A_n^G(R_{d^*}) \to A_n^G(R_{e^*}^c)$ is surjective. Choose an inverse image $\sum_{d^*} \beta_{d^*}$ of $\gamma_{e^*}$ under this map. Consider the commutative diagrams
		\begin{center}
			\begin{tikzpicture}[description/.style={fill=white,inner sep=2pt}]
				\matrix(m)[matrix of math nodes, row sep=1.5em, column sep=3em, text height=1.5ex, text depth=0.25ex]
				{
					W_{d^*}^o \times^{P_{d^*}} G & W_{d^*} \times^{P_{d^*}} G & R_d \times^{P_{d^*}} G \\
					R_{d^*}^{\HN} & R_{d^*} & R_d \\
				};
				\path[->, font=\scriptsize]
				(m-1-1) edge 	(m-1-2)
				(m-1-1) edge 	(m-2-1)
				(m-1-2) edge 	(m-2-2)
				(m-2-1) edge 	(m-2-2)
				(m-1-3) edge 	(m-2-3)
				(m-1-2) edge 	(m-1-3)
				(m-2-2) edge	(m-2-3);
			\end{tikzpicture}
		\end{center}
		induced by the natural morphism $R_d \times^{P_{d^*}} G \to R_d$. The left squares in the above diagram is cartesian. The uniqueness of the HN filtration implies that $W_{d^*}^o \times^{P_{d^*}} G \to R_{d^*}^{\HN}$ is an isomorphism. Denote by $W_{d^*}^c$ the complement of $W_{d^*}^o$ in $W_{d^*}$. The cartesian diagram
		\begin{center}
			\begin{tikzpicture}[description/.style={fill=white,inner sep=2pt}]
				\matrix(m)[matrix of math nodes, row sep=1.5em, column sep=3em, text height=1.5ex, text depth=0.25ex]
				{
					W_{d^*}^c \times^{P_{d^*}} G & W_{d^*} \times^{P_{d^*}} G\\
					R_{d^*}^c & R_{d^*} \\
				};
				\path[->, font=\scriptsize]
				(m-1-1) edge 	(m-1-2)
				(m-1-1) edge 	(m-2-1)
				(m-1-2) edge node[auto] {$\pi^{d^*}$} 	(m-2-2)
				(m-2-1) edge node[auto] {$i^{d^*}$}		(m-2-2);
			\end{tikzpicture}
		\end{center}
		induces an exact sequence
		$$
			A_n^G(W_{d^*}^c \times^{P_{d^*}} G) \to A_n^G(R_{d^*}^c) \oplus A_n^G(W_{d^*} \times^{P_{d^*}} G) \to A_n^G(R_{d^*}) \to 0
		$$
		using \cite[Example 1.8.1]{Fulton:98}. The surjection in the above sequence is given by mapping $\gamma_{d^*} + \alpha_{d^*,d^*}$ to $i_*^{d^*}(\gamma_{d^*}) + \pi_*^{d^*}(\alpha_{d^*,d^*})$. Let $\gamma_{d^*} + \alpha_{d^*,d^*}$ be an inverse image of $\beta_{d^*}$. By the induction assumption, $\psi_{d^*}$ is surjective, whence there exist $\alpha_{f^*,d^*}$ for $f^* \rhd d^*$ such that $\gamma_{d^*} = \sum_{f^*} \psi_{f^*,d^*}(\alpha_{f^*,d^*})$. This implies that
		$$
			\gamma_{e^*} = \sum_{d^*} \Big( \psi_{d^*,e^*}(\alpha_{d^*,d^*}) + \sum_{f^* \rhd d^*} \psi_{f^*,d^*}(\alpha_{f^*,d^*}) \Big) = \sum_{f^*} \psi_{f^*,e^*} \Big( \sum_{d^* \unlhd f^*} \alpha_{f^*,d^*} \Big)
		$$ 
		which completes the proof.
	\end{proof}

	\begin{rem}
		Unfortunately, the conditions of Lemma \ref{kack_lemma} are rarely fulfilled. In \cite[Section 3]{Reineke:03}, Reineke shows that for the example of a linearly oriented $A_3$ quiver with dimension vector $(1,1,1)$, the assumption of the lemma does not hold. However, we can apply the above result to, for example, the quiver $\tildeAOne$ (cf.\ Proposition \ref{tensorProd}). Additionally, the ``relative'' HN stratification of smooth models fulfills this requirement, which enables us to prove Theorem \ref{main_thm} with the same methods as Lemma \ref{kack_lemma}.
	\end{rem}

	\subsection*{Generating Series}

	Fix a finite field $\F = \F_q$. We define the \textbf{completed Hall algebra} of a quiver $Q$ (which is a completed version of the Hall algebra as in \cite{Ringel:90} and \cite{Green:95}) as the vector space
	$$
		H = H((Q)) = \Big\{ f \mid f: \bigsqcup_{d \in \Z_{\geq 0}^{Q_0}} R_d(\F)/G_d(\F) \to \Q \Big\}
	$$
	equipped with the following multiplication: For two functions $f$ and $g$, we define their product $f \circ g$ by
	$
		(f \circ g)(X) = \sum_{U \sub X} f(U) g(X/U).
	$
	The sum ranges over all subrepresentations $U$ of $X$ over $\F$ and is thus finite. We can see that $H((Q))$ becomes an associative algebra. Let's assume that the quiver $Q$ is symmetric (the results below can also be stated for non-symmetric quivers, but they are a little more technical then). Reineke has shown in \cite{Reineke:06} that the map $\int: H \to \Q(q^{1/2})[[t_i \mid i \in Q_0]]$ defined by
	$$
		\int f = \sum_{[X]} \frac{(-q^{1/2})^{\chi( \dimvect X, \dimvect X )}}{\sharp \Aut(X)} \cdot f(X) \cdot t^{\dimvect X}
	$$
	is a homomorphism of algebras. Let $\one \in H$ be the constant function with value one (this isn't the unit element of $H$). Then we compute
	$$
		\int \one = \sum_{d} (-q^{1/2})^{\chi( d,d )} \frac{\sharp R_d(\F)}{\sharp G_d(\F)} t^d = \sum_d (-q^{1/2})^{-\chi( d,d )} \prod_i \prod_{\nu = 1}^{d_i} (1-q^{-\nu})^{-1} t^d
	$$
	which is precisely $P_Q(q^{-1},t)$, the generating series of the CoHa $\HH(Q)$ (with the $(\Gamma \times \Z)$-grading) evaluated at $q^{-1}$. For a stability condition $\theta$ and a rational number $\mu$, we define $\one^{\sst,\mu} = \one^{\theta-\sst,\mu} \in H$ by
	$$
		\one^{\sst,\mu}(M) = \begin{cases} 1 & \text{if } M \text{ is } \theta\text{-semi-stable of slope }\mu \\ 0 & \text{otherwise}. \end{cases}
	$$
	Integrating this function, we obtain
	$$
		\int \one^{\sst,\mu} = \sum_{d} (-q^{1/2})^{\chi( d,d )} \frac{\sharp R_d^{\sst}(\F)}{\sharp G_d(\F)} t^d.
	$$
	Using the Harder--Narasimhan stratification described on page \pageref{HN}, Reineke has shown in \cite{Reineke:03}:

	\begin{thm}[{\cite[Proposition 4.12]{Reineke:03}}] \label{reineke}
		In $H((Q))$, we have the identity $\one = \prod\limits_{\mu \in \Q}^{\ot} \one^{\sst,\mu}$.
	\end{thm}

	The infinite product above is defined as the series $\sum_{r \geq 1} \sum_{\mu^1 > \ldots > \mu^r} \one^{\sst,\mu^1} \circ \ldots \circ \one^{\sst,\mu^r}$. Being a continuous homomorphism of algebras, the map $\smash{\int}$ preserves this identity. In \cite[5.2]{KS:11}, Kontsevich--Soibelman asked the following question: Given a quiver $Q$ and a stability condition $\theta$, do there exist embeddings $\HH^{\sst,\mu} \into \HH$ of algebras for every $\mu$ such that the induced map $\bigotimes_{\mu \in \Q}^{\ot} \HH^{\sst,\mu} \to \HH$ (by multiplication from left to right) is an isomorphism? They gave a positive answer for the CoHa of $A_2$. In \cite{Rimanyi:13}, Rim{\'a}nyi showed that this is true for all Dynkin quivers except for type $E_8$. We give another example where such an isomorphism exists.

	\subsection*{An Example}

	Like in Example \ref{exC_first}, we consider the quiver $\tildeAOne$ and let $d = (m \rla n)$ be a dimension vector. A representation of $\tildeAOne$ of this dimension vector consists of a pair of matrices $A \in M_{n \times m}(\C)$ and $B \in M_{m \times n}(\C)$. On the vector space $R_d = M_{n \times m} \oplus M_{m \times n}$, we have the action of the linear algebraic group $G_d = \Gl_m \times \Gl_n$ via
	$$
		(g,h) * (A,B) = (hAg^{-1},gBh^{-1}).
	$$
	Fix the stability condition $\theta = (1,-1)$. The associated slope function is given by $\slope(m \rla n) = \frac{m-n}{m+n}$. The largest possible slope is $1$, for dimension vectors $m \rla 0$, while the smallest slope is $-1$, which we obtain for dimension vectors $0 \rla n$.

	\label{secondEx}

	For $d = (m \rla 0)$, there is a unique representation of this dimension vector and this one is $\theta$-semi-stable (let's omit $\theta$ in the following). Likewise, we can see that $R_{0 \rla n} = R_{0 \rla n}^{\sst} = \pt$. Let's consider dimension vectors of slope $0$, i.e.\ $d = (m \rla m)$. A representation $(A,B)$ is semi-stable if for two linear subspaces $U, V \sub \C^m$ satisfying $AU \sub V$ and $BV \sub U$, it follows that $\dim U \leq \dim V$. It is easy to see that this is equivalent to requiring $A$ to be invertible. Thus
	$$
		R_{m \rla m}^{\sst} = \Gl_m \times M_{m}.
	$$	
	For a dimension vector $d = (m \rla n)$ of slope neither $1,0$ nor $-1$ (this means $m \neq n$ and both non-zero) there are no semi-stable representations: Let $(A,B)$ be a representation of dimension vector $d$. We analyze the case $m > n$ (for $m < n$ we may argue in a similar way). The matrix $A$ must have a non-trivial kernel whence $(\ker A, 0)$ provides a subrepresentation of $(A,B)$ of slope $1 > \slope(d)$.

	We described the CoHa of $\tildeAOne$ in Example \ref{exC_first}. We see at once that the subalgebras $\HH^{1} = \bigoplus_n \HH_{n \rla 0}$ and $\HH^{-1} = \bigoplus_n \HH_{0 \rla n}$ are both isomorphic to the CoHa of the quiver which consists of a single point (without arrows). This isn't surprising as $[R_{n \rla 0}/G_{n \rla 0}] = [R_{0 \rla n}/G_{0 \rla n}] = [\pt/\Gl_n]$. But the unique representation of dimension vector $n \rla 0$ (resp. $0 \rla n$) is semi-stable, thus
	$$
		\HH^{1} = \HH^{\sst,1} \cong \HH(\ptquiv) \cong \HH^{\sst,-1} = \HH^{-1}.
	$$
	Now, let $d = (m \rla m)$ be a dimension vector of slope $0$. The morphism $M_{m} \to R_{m \rla m}$ which sends a matrix $B$ to $(E,B)$---here $E$ is the $(m \times m)$-unit matrix---is $\Gl_m$-equivariant when viewing $\Gl_m$ as a subgroup of $G_{m \rla m} = \Gl_m \times \Gl_m$ via the diagonal embedding. By the above considerations, this map factors through $R_{m \rla m}^{\sst} = \Gl_m \times M_{m}$ and induces an isomorphism $[M_{m}/\Gl_m] \to [R_{m \rla m}^{\sst}/G_{m \rla m}]$ \label{secondEx2} of moduli stacks. Passing to equivariant Chow rings (or equivariant cohomology rings; in this case, it doesn't make any difference), we obtain homomorphisms
	\begin{center}
		\begin{tikzpicture}[description/.style={fill=white,inner sep=2pt}]
			\matrix(m)[matrix of math nodes, row sep=2em, column sep=2em, text height=1.5ex, text depth=0.25ex]
			{
				A_{G_{m \rla m}}^*(R_{m \rla m}) &	A_{G_{m \rla m}}^*(R_{m \rla m}^{\sst}) & A_{\Gl_m}^*(M_{m}) \\
				\Q[x_1,\ldots,x_m,y_1,\ldots,y_m]^{S_m \times S_m}	&& \Q[t_1,\ldots,t_m]^{S_m}. \\
			};
			\path[->, font=\scriptsize]
			(m-1-1) edge (m-1-2)
			(m-1-2) edge node[auto]{$\cong$} (m-1-3)
			(m-2-1) edge (m-2-3);
			\draw[-] ($(m-1-1.south) + (-.1em,0)$) -- ($(m-2-1.north) + (-.1em,.2em)$);
			\draw[-] ($(m-1-1.south) + (.1em,0)$) -- ($(m-2-1.north) + (.1em,.2em)$);
			\draw[-] ($(m-1-3.south) + (-.1em,0)$) -- ($(m-2-3.north) + (-.1em,.2em)$);
			\draw[-] ($(m-1-3.south) + (.1em,0)$) -- ($(m-2-3.north) + (.1em,.2em)$);
		\end{tikzpicture}
	\end{center}
	The lower map is given by sending $x_i$ and $y_i$ to $t_i$. Observing that the Hecke correspondences are compatible with these maps, we have identified the semi-stable ChowHa/CoHa of $Q$ of slope $0$ with the CoHa of the Jordan quiver $\loopquiv$, i.e.
	$$
		\HH^{0} \onto \HH^{\sst,0} \cong \HH(\loopquiv).
	$$
	But $\HH(\loopquiv)$ is isomorphic to the symmetric algebra over $\HH_1(\loopquiv) = \Q[t]$ (as seen in \ref{exB}), i.e.\ a polynomial algebra in infinitely many variables $\phi_0, \phi_1, \phi_2, \ldots$ with $\phi_i(t) = t^i$. A surjection to a polynomial algebra allows a section (by choosing inverse images of the variables). So here, sending $\phi_i$ to $x^i \in \HH_{1 \rla 1}(\tildeAOne) = \Q[x,y]$ yields a section $\Psi^0$ of the surjection $\HH^0(\tildeAOne) \onto \HH(\loopquiv)$.

	\begin{rem*}
		Note that $\Psi^0$ can also be obtained geometrically as the pull-back of the morphism $R_{m \rla m} \to M_{m}$ mapping $(A,B)$ to $BA$ which is $G_{m \rla m}$-equivariant via the projection $G_{m \rla m} = \Gl_m \times \Gl_m \to \Gl_m$ to the first factor.
	\end{rem*}

	We obtain a homomorphism of algebras
	$
		\Psi: \HH(\ptquiv) \otimes \HH(\loopquiv) \otimes \HH(\ptquiv) \to \HH(\tildeAOne)
	$
	by mapping $\Psi(f \otimes g \otimes h) = f * \Psi^0(g) * h$. The tensor product on the left-hand side is to be understood as a super-commutative tensor product which explains the algebra structure. We show:

	\begin{prop} \label{tensorProd}
		The map $\Psi: \HH(\ptquiv) \otimes \HH(\loopquiv) \otimes \HH(\ptquiv) \to \HH(\tildeAOne)$ is an isomorphism of algebras.
	\end{prop}

	\begin{proof}
		We prove that $\Psi$ is surjective. It is clear that $\HH_{m \rla 0}$ and $\HH_{0 \rla n}$ are contained in $\im \Psi$. Let $d = (m \rla n)$ be a dimension vector with $m \neq 0$ and $n \neq 0$. The possible HN types (cf.\ page \pageref{HN}) of $d$ are of the form $d^*(k) = (m-k \rla 0, k \rla k, 0 \rla n-k)$ with $k = 0,\ldots, \min\{m,n\}$ (for $k=0$ or $k = \min\{m,n\}$, the dimension vectors in $d^*(k)$ which are $0$ are understood to be omitted). The associated HN stratum is
		$$
			R_{d^*(k)}^{\HN} = \{ (A,B) \mid \rk A = k \}
		$$
		whence the closures of the HN strata give a filtration $R_d = R_{d^*(m)} \supseteq R_{d^*(m-1)} \supseteq \ldots \supseteq R_{d^*(0)}$. Lemma \ref{kack_lemma} therefore applies here, yielding that
		$$
			\HH_{m \rla n} = \begin{cases}
						\sum_{r \geq 2} \sum_{d^* \in \HN_r(m \rla n)} \HH_{d^1} * \ldots * \HH_{d^r} & \text{ for } m \neq n \\
						\im \Psi^0 \oplus \sum_{r \geq 2} \sum_{d^* \in \HN_r(m \rla m)} \HH_{d^1} * \ldots * \HH_{d^r} & \text{ for } m = n.
			                 \end{cases}
		$$
		In order to show that $\Psi$ is moreover an isomorphism of algebras, it remains to verify that the generating series of both algebras agree. We consider the elements $\one, \one^{\sst,1}, \one^{\sst,0}$ and $\one^{\sst,-1}$ of $H((\tildeAOne))$. We have shown that the following identities in the power series algebra $\Q(q^{1/2})[[x,y]]$ hold:
		\begin{align*}
			\int \one &= P_{\tildeAOne}(q^{-1},x,y) &
			\int \one^{\sst,1} &= P_{\raisebox{-.2em}\bull}(q^{-1},x) \\
			\int \one^{\sst,0} &= P_{\loopquiv}(q^{-1},xy) &
			\int \one^{\sst,-1} &= P_{\raisebox{-.2em}\bull}(q^{-1},y).
		\end{align*}
		Thus, we deduce from Reineke's theorem that $\Psi$ is in fact an isomorphism.
	\end{proof}

	Put $\psi_i^+ \in \HH_{1 \rla 0} = \Q[x]$ to be $\psi_i^+(x) = x^i$, and define $\psi_i^- \in \HH_{0 \rla 1} = \Q[y]$ by $\psi_i^-(y) = y^i$. Abusing notation, we denote by $\phi_i \in \HH_{1 \rla 1}$ the function $\phi_i(x,y) = x^i$. From Proposition \ref{tensorProd}, we deduce:

	\begin{cor} \label{freeSuper}
		The CoHa of $\tildeAOne$ is the free super-commutative algebra over the graded (super\nobreakdash-)vector space $V$ spanned by
		$$
			\psi_0^+, \psi_1^+, \ldots, \phi_0, \phi_1, \ldots, \psi_0^-, \psi_1^-, \ldots
		$$
		and as a graded vector space, $V$ is isomorphic to $V^{\mathrm{prim}} \otimes \Q[z]$ with $z$ living in bidegree $(0 \rla 0,2)$ and $V^{\mathrm{prim}} = \Q \cdot \psi_0^+ \oplus \Q \cdot \phi_0 \oplus \Q \cdot \psi_0^-$.
	\end{cor}
	
	\section{Modules from Smooth Models}

	\newcommand{\fr}{\mathrm{fr}}

	We are going to construct modules over $\HH^{\sst,\mu}$ (resp. $\AA^{\sst,\mu}$) which come from smooth models. These modules were introduced, in the more general context of framed stable objects, in Section 4 of Soibelman's paper \cite{Soibelman:14}.

	\subsection*{Smooth Models}

	We assume, again, a quiver $Q$, a stability condition $\theta$ of $Q$, and a rational number $\mu$ to be fixed. Let $n$ be a dimension vector for $Q$. It will be the framing datum. Construct a new quiver $\hat{Q}(n)$ by adding an extra vertex $\infty$ to the vertexes of $Q$ and having $n_i$ arrows pointing from $\infty$ to $i$ for every $i \in Q_0$. This framed quiver depends on $n$; however, when the dependency on $n$ can be neglected, we will drop it in the notation for the sake of brevity. Let $\epsilon > 0$ be a rational number. Define a stability condition $\hat{\theta} = \hat{\theta}_\epsilon$ for $\hat{Q}$ by letting $\hat{\theta}(i) = \theta(i)$ for all $i \in Q_0$ and
	$$
		\hat{\theta}(\infty) = \mu + \epsilon.
	$$
	Let $d$ be a dimension vector of $Q$ and let $\hat{d}$ be the dimension vector of $\hat{Q}$ with $\hat{d}_i = d_i$ and $\hat{d}_\infty = 1$. Suppose that $\slope(d) = \mu$. A representation $(M,f)$ of $\hat{Q}$ of dimension vector $\hat{d}$ consists of a representation $M$ of $Q$ and linear maps $f_i: \C^{n_i} \to \C^{d_i}$. Engel--Reineke have shown:

	\begin{prop}[{\cite[Proposition 3.3]{ER:09}}] \label{engel-reineke}
		If $\epsilon$ is small enough then, for a representation $(M,f)$ of $\hat{Q}$, the following are equivalent:
		\begin{enumerate}
			\item $(M,f)$ is $\hat{\theta}$-semi-stable.
			\item $(M,f)$ is $\hat{\theta}$-stable.
			\item $M$ is $\theta$-semi-stable and $\slope(M') < \slope(M)$ for every proper subrepresentation $M'$ of $M$ which contains the image of $f$.
		\end{enumerate}
	\end{prop}

	For every dimension vector $d$ of $Q$ of slope $\mu$, choose an $\epsilon = \epsilon_d > 0$ such that the above equivalences are valid \emph{for all sub-dimension vectors} $d' \leq d$. Denote then
	\begin{align*}
		&\hat{R}_d(Q,n) = R_{\hat{d}}\big( \hat{Q}(n) \big) = R_d(Q) \oplus \overbrace{\bigoplus_{i \in Q_0} \Hom(\C^{n_i},\C^{d_i})}^{=: F_d(n)} \\ 
		&\hat{R}_d^{\theta-\st}(Q,n) = R_{\hat{d}}^{\hat{\theta}_{\epsilon_d}-\mathrm{(s)st}}\big( \hat{Q}(n) \big).
	\end{align*}
	Although it is actually $\hat{\theta}$-stability, we write $\smash{\hat{R}_d^{\theta-\st}(Q,n)}$ instead. On $\hat{R}_d = \hat{R}_d(Q,n)$, there is a natural action of the group $G_d$ (the usual action on $R_d$ and left-multiplication on the framing). The induced action on $\hat{R}_d^{\st} = \hat{R}_d^{\theta-\st}(Q,n)$ is free. The geometric quotient 
	$$
		\hat{M}_d^{\theta-\st}(Q,n) := \hat{R}_d^{\theta-\st}(Q,n)/G_d
	$$
	(which exists by Mumford's GIT) is smooth and the natural map $\hat{M}_d^{\theta-\st}(Q,n) = \hat{M}_d^{\st} \to M_d^{\sst} = R_d^{\sst}/PG_d$ is projective.

	\begin{defn*}
		The variety $\hat{M}_d^{\theta-\st}(Q,n)$ is called a \textbf{smooth model} for $M_d^{\theta-\sst}(Q)$.
	\end{defn*}

	\subsection*{Modules over the Semi-Stable CoHa}

	Like we did before, whenever no confusion can arise we will suppress the dependency on $Q$, $\theta$ and $n$ in the notation. Let $d$ and $e$ be dimension vectors of slope $\mu$. We define, in accordance to the notation from the first section,
	\begin{align*}
		\hat{\uppertri} &:= R_{\destar}(Q) \oplus F_{d+e}(n) \\
		\hat{\uppertri}{}^{\st} &:= \hat{\uppertri} \cap \hat{R}_{d+e}^{\st}
	\end{align*}
	For $(M,f) = \smash{\Big( \big( \begin{smallmatrix} M' & * \\ & M'' \end{smallmatrix} \big), \big( \begin{smallmatrix} f' \\ f'' \end{smallmatrix} \big) \Big)}$ belonging to $\hat{\uppertri}{}^{\st}$, let $U''$ be a proper subrepresentation of $M''$ which contains the image of the linear map $f''$. Consider the short exact sequence
	$$
		0 \to M' \to M \to M'' \to 0
	$$
	and form the pull-back $U = M \times_{M''} U''$ which is a proper subrepresentation of $M$. We obtain a short exact sequence $0 \to M' \to U \to U'' \to 0$ and we see that $\im f$ is contained in $U$. As $(M,f)$ is (semi-)stable, $\slope(U) < \slope(M) = \mu$ and thus, $\slope(M'') = \mu = \slope(M') > \slope(U'')$. Hence, $(M'',f'')$ is also $\hat{\theta}_{\epsilon_{d+e}}$-stable (which is equal to $\hat{\theta}_{\epsilon_{e}}$-stability for a representation of dimension vector $e$). This yields maps
	$$
		R_d^{\sst} \times \hat{R}_{e}^{\st} \ot \hat{\uppertri}{}^{\st} \to \hat{R}_{d+e}^{\st}
	$$
	the first being flat, the latter a closed embedding. Note that the varieties and maps are independent of the choices of the $\epsilon$'s. Putting
	\begin{align*}
		&\hat{\HH}^{\theta-\st,\mu}(Q,n) = \bigoplus_{d \in \Gamma_\mu} \hat{\HH}_d^{\theta-\st}(Q,n) \text{ with} \\
		&\hat{\HH}_d^{\theta-\st}(Q,n) = H^*_{G_d}(\hat{R}^{\theta-\st}_d(Q,n)) = H^*(\hat{M}^{\theta-\st}_d(Q,n)),
	\end{align*}
	we have declared an $\HH^{\sst,\mu}$-module structure on $\hat{\HH}^{\st,\mu}$. It is graded by $\Gamma_\mu$. There are natural morphisms $R_d^{\sst} \ot \smash{\hat{R}_{d}^{\st}}$ which give rise to a map 
	$$
		\HH^{\sst,\mu} \to \hat{\HH}^{\st,\mu}.
	$$
	This map is $\HH^{\sst,\mu}$-linear (this can be shown in exactly the same way as \cite[Proposition 3.3]{Franzen:13:NCHilb_Loop}).

	\subsection*{Modules over the Semi-Stable ChowHa}

	As for the CoHa and the semi-stable CoHa, we can use intersection theory to construct a module $\hat{\AA}^{\st,\mu}$ over the semi-stable ChowHa arising from smooth models. We employ the same Hecke correspondences as in the cohomological version. The natural map $\AA^{\sst,\mu} \to \hat{\AA}^{\st,\mu}$ is $\AA^{\sst,\mu}$-linear. We prove:

	\begin{thm} \label{main_thm}
		The homomorphism $\AA^{\theta-\sst,\mu}(Q) \to \hat{\AA}^{\theta-\st,\mu}(Q,n)$ is surjective and the kernel equals
		$$
			\sum_{p,q} \AA^{\sst}_p * (e_q^n \cup \AA^{\sst}_q),
		$$
		the sum running over all $p,q$ with $\slope(p) = \slope(q) = \mu$ and $q \neq 0$. Here, $\cup$ denotes the intersection product in $\AA^{\sst}_q = A^*_{G_q}(R_q^{\sst})$ and $e_q^n$ is the product of equivariant Chern classes
		$$
			e_q^n = \prod_i c_{q_i}^{G_q}(R_q^{\sst} \times \C^{q_i} \to R_q^{\sst})^{n_i},
		$$
		where $G_q$ is understood to act on $\C^{q_i}$ by left multiplication of the $i$\textsuperscript{th} factor.
	\end{thm}

	\begin{proof}
		Fix a dimension vector $d$. The map $\phi: \hat{R}_d^{\st} \to R_d^{\sst}$ is the composition of the open embedding $\smash{\hat{R}_d^{\st}} \into R_d^{\sst} \times F_d$ and the (equivariant) vector bundle $R_d^{\sst} \times F_d \to R_d^{\sst}$. Therefore, $\phi$ is flat and the induced map in equivariant intersection theory is surjective. As a vector bundle induces an isomorphism in intersection theory, it suffices to compute a presentation of 
		$$
			A_m^G(R_d^{\sst} \times F_d) \to A_m^G(\hat{R}_d^{\st}).
		$$
		The unstable locus of $\hat{R}_d$ consists of those $(M,f)$ such that $M$ is either unstable or $M$ is semi-stable but there exists a proper subrepresentation $M'$ of $M$ of the same slope which contains $\im f$. Denote the subset of pairs $(M,f)$ satisfying the latter with $\hat{R}_d^x$. The sequence
		$$
			A_m^G(\hat{R}_d^x) \to A_m^G(R_d^{\sst} \times F_d) \to A_m^G(\hat{R}_d^{\st}) \to 0
		$$
		is exact. For a framed representation $(M,f) \in \hat{R}_d^x$, define $L(M,f)$ to be the subrepresentation of $M$ which is minimal among those containing the image of $f$ and having the slope $\mu$. Let $p = \dimvect L(M,f)$. As $U := L(M,f)$ is a proper subrepresentation of $M$, we get that $\hat{\slope}(\hat{p}) > \hat{\slope}(\hat{d})$ (using \cite[Lemma 3.2]{ER:09}). Moreover $(U,f)$ is framed stable, as $U$ is semi-stable and by the defining minimality condition. As $M/U$ is also semi-stable, say of dimension vector $q$, the HN-filtration of $(M,f)$ is
		$$
			(M,f) \supset (U,f) \supset 0
		$$
		whose type is $(\hat{p},q)$. The HN stratum of this type is thus the set of all $(M,f)$ with $M$ semi-stable and $\dimvect L(M,f) = p$. We show that the closure of the stratum $\hat{R}_{(\hat{p},q)}^{\HN}$ in $R_d^{\sst} \times F_d$ is
		$$
			Z_p := \{ (M,f) \mid M \text{ semi-stable, } \dimvect L(M,f) \leq p \}.
		$$
		Denote $\Gr_p^d = \prod_i \Gr_{p_i}^{d_i}$ and consider the morphisms
		$$
			\Gr_p^d \xot{}{\pi_3} R_d^{\sst} \times F_d \times \Gr_p^d \xto{}{\pi_{12}} R_d^{\sst} \times F_d.
		$$
		Let $Y_p$ be the closed subset of all $(M,f,U) \in R_d^{\sst} \times F_d \times \Gr_p^d$ such that $U$ determines a subrepresentation of $M$ which contains the image of $f$. Then $Z_p$ equals $\pi_{12}(Y_p)$ whence it is closed as $\pi_{12}$ is proper. Now, consider the restriction of $\pi_3$ to $Y_p$. The fiber of $U_0$ which is given by the subspaces $\C^{p_i} \sub \C^{d_i}$ generated by the first $p_i$ coordinate vectors is
		$$
			(Y_p)_{U_0} = \uppertripq^{\sst} \times \fpo
		$$
		and therefore irreducible. As $Y_p \to \Gr_p^d$ is an equivariant morphism to a homogeneous space such that one fiber is (and hence all fibers are) irreducible, $Y_p$ must itself be irreducible. We have deduced that $Z_p$ is closed and irreducible. As
		$$
			\hat{R}_{(\hat{p},q)}^{\HN} = Z_p - \bigcup_{\substack{p' < p \\ \slope(p') = \mu}} Z_{p'},
		$$
		$Z_p$ is indeed the closure of the HN stratum in $R_d^{\sst} \times F_d$ and coincides with the union of HN strata $\bigcup_{p'} \smash{\hat{R}_{(\hat{p},q)}^{\HN}}$ over all $p' \leq p$ with $\slope(p') = \mu$. Arguing in the same way as in the proof of Lemma \ref{kack_lemma}, we obtain the exactness of the sequence
		$$
			\bigoplus_{\substack{q \neq 0 \\ \slope(q) = \mu}} A_n^G \bigg( \Big( \uppertripq^{\sst} \times \fpo \Big) \times^{P_{p,q}} G \bigg) \to A_n^G(R_d^{\sst} \times F_d) \to A_n^G(\hat{R}_d^{\st}) \to 0.
		$$
		In this context, $G = G_d$ and $P_{p,q} = \big( \begin{smallmatrix} G_p & * \\ & G_q \end{smallmatrix} \big)$. The left hand map of this sequence comes from the composition 
		$$
			\uppertripq^{\sst} \times \fpo \xto{}{s_0} \uppertripq^{\sst} \times F_d \to R_d^{\sst} \times F_d
		$$
		and $s_0$ is a base extension of the zero section of the $G_q$-equivariant vector bundle $R_q^{\sst} \times F_q \to R_q^{\sst}$. This $G_q$-linear vector bundle can be seen as
		$$
			\bigoplus_i V(q_i)^{\oplus n_i},
		$$
		with $V(q_i)$ being the trivial vector bundle of rank $q_i$ and $G_q$ acting by left multiplication of the $i$\textsuperscript{th} factor. Therefore, the push-forward $(s_0)_*$ is precisely the multiplication with the product $\prod_i \smash{c_{q_i}^{G_q}}(V(q_i))^{n_i}$ as asserted.
	\end{proof}

	For two framing data $n$ and $m$ satisfying $n \leq m$ (i.e.\ $n_i \leq m_i$ for every $i$), we choose $\C^{n_i}$ to be embedded into $\C^{m_i}$ as the coordinate space of the first $n_i$ unit vectors. The natural linear map $\hat{R}_d(Q,n) \to \hat{R}_d(Q,m)$ is then $G_d$-equivariant and respects stability (with a suitable choice of $\epsilon$). Thus, we get a map $A^k\big( \smash{\hat{M}^{\st}(Q,m)} \big) \to A^k\big( \smash{\hat{M}^{\st}(Q,n)} \big)$. We put
	$$
		\overleftarrow{\AA}_d^{\theta-\sst}(Q) = \bigoplus_k\ \varprojlim_n A^k\big( \hat{M}^{\st}(Q,n) \big)
	$$
	for every $d \in \Gamma_\mu$. Let $\overleftarrow{\AA}^{\theta-\sst,\mu}(Q) = \bigoplus_{d \in \Gamma_\mu} \overleftarrow{\AA}_d^{\theta-\sst}(Q)$. From the above theorem, we obtain:

	\begin{cor}
		The natural map $\AA^{\theta-\sst,\mu}(Q) \to \overleftarrow{\AA}^{\theta-\sst,\mu}(Q)$ is an isomorphism.
	\end{cor}

	Note that the projective limits  $\overleftarrow{\AA}^{\sst,\mu}$ and $\overleftarrow{\HH}^{\sst,\mu}$ coincide, provided that the quiver is acyclic. This follows from a result due to King and Walter (cf.\ \cite[Theorem 3]{KW:95}), which asserts that the cycle map for a quiver moduli is an isomorphism if the quiver is acyclic and if stability and semi-stability coincide. Both conditions are satisfied for smooth models if the original quiver $Q$ itself is acyclic.
	
	\section{Application to Non-Commutative Hilbert Schemes}

	For a quiver $Q$, a framing datum $n$, and a dimension vector $d$, the smooth model $\hat{M}_d^{\st}$ attached to the stability condition $\theta = 0$ (and the slope $\mu = 0$) is called the \textbf{non-commutative Hilbert scheme} and is commonly denoted $\Hilb_d = \Hilb_{d,n}(Q)$. It arises as the geometric quotient $\smash{\hat{R}_d^{\st}/G_d}$ and by Proposition \ref{engel-reineke}, a point of $\smash{\hat{R}_d}$ is a pair $(M,f)$ consisting of a representation $M$ of $Q$ and linear maps $f_i: \C^{n_i} \to \C^{d_i}$ such that no proper subrepresentation of $M$ contains the image of $f$. Proposition 7.8 of \cite{ER:09} shows that $\Hilb_d$ possesses a cell decomposition, i.e.\ it has a filtration by closed subsets such that the successive complements are isomorphic to affine spaces. A formal consequence of the existence of a cell decomposition is that the cycle map $A_*(\Hilb_d) \to H_*(\Hilb_d)$ is an isomorphism (cf.\ \cite[Example 19.1.11]{Fulton:98}). It hence does not matter if we consider the intersection theoretic or 
	the cohomological version of the CoHa-module which is induced by the non-commutative Hilbert schemes. In this context, Theorem \ref{main_thm} which establishes a connection between the CoHa of $Q$ and the module $\hat{\HH}(Q,n) = \hat{\HH}^{0-\st}(Q,n)$ reads as follows:

	\begin{cor} \label{cyclicHilb}
		The homomorphism $\HH(Q) \to \hat{\HH}(Q,n)$ is surjective and the kernel equals the sum
		$$
			\sum_{p,q \geq 0,\ q \neq 0} \HH_p * (e_q^n \cup \HH_q).
		$$
	\end{cor}

	\subsection*{Examples}

	This result has been worked out in the case of the $r$-loop quiver in \cite[Theorem 3.6]{Franzen:13:NCHilb_Loop}. The cases $r=0$ and $r=1$ can be displayed using the identification as an exterior algebra and a symmetric algebra, respectively:

	\begin{ex}
		\begin{enumerate}
			\item Identifying $\HH(\ptquiv) \cong \bigwedge(\psi_0,\psi_1,\ldots)$, the CoHa-module $\hat{\HH}(\ptquiv,n)$ is isomorphic to the exterior algebra $\bigwedge(\psi_0,\psi_1,\ldots,\psi_{n-1})$. This has also been worked out in \cite[Proposition 3.3]{Xiao:14}.
			\item The same argument applies for the Jordan quiver $\loopquiv$. Under the isomorphism $\HH(\loopquiv) \cong \Q[\psi_0,\psi_1,\ldots]$, the CoHa-module $\smash{\hat{\HH}(\loopquiv,n)}$ corresponds to $\Q[\psi_0,\psi_1,\ldots,\psi_{n-1}]$.
		\end{enumerate}
	\end{ex}

	\begin{ex}
		Let $Q$ be the quiver $\tildeAOne$. We fix a framing datum $r \rla s$. For a dimension vector $m \rla n$, a framed representation is a quadruple $(A,B,C,D)$ consisting of $A \in M_{n \times m}$, $B \in M_{m \times n}$, $C \in M_{m \times r}$, and $D \in M_{n \times s}$. Such a representation is $\hat{0}$-stable if for all linear subspaces $U \sub \C^m$ and $V \sub \C^n$ such that
		\begin{align}
			AU &\sub V, \\
			BV &\sub U, \\
			\im C &\sub U, \text{ and}\\
			\im D &\sub V,
		\end{align}
		it follows that $U = \C^m$ and $V = \C^n$.
		Now, fix the stability condition $\theta = (1,-1)$. We have seen on page \pageref{secondEx} that $\theta$-semi-stable representations exist only in slopes $1$, $0$, and $-1$ and that the semi-stable CoHas (or ChowHas, they are isomorphic in this case) identify with
		$$
			\HH^{1,\sst}(\tildeAOne) \cong \HH(\ptquiv),\hphantom{ and } \HH^{0,\sst}(\tildeAOne) \cong \HH(\loopquiv), \text{ and } \HH^{-1,\sst}(\tildeAOne) \cong \HH(\ptquiv).
		$$
		Moreover, we have constructed an isomorphism of $(\Gamma \times \Z)$-graded super-commutative algebras $\Psi: \HH(\ptquiv) \otimes \HH(\loopquiv) \otimes \HH(\ptquiv) \to \HH(\tildeAOne)$. We want to compare the respective modules arising from non-commutative Hilbert schemes with respect to this isomorphism. Therefore, let's describe $\hat{\theta}$-stability in this particular case. According to Proposition \ref{engel-reineke}, a framed representation $(A,B,C,D)$ as above is $\hat{\theta}$-stable if and only if $(A,B)$ is $\theta$-semi-stable and for all subspaces $U \sub \C^m$ and $V \sub \C^n$ satisfying conditions (1) to (4), it follows that either $U = \C^m$ and $V = \C^n$ or
		$$
			\frac{\dim U - \dim V}{\dim U + \dim V} < \frac{m-n}{m+n}.
		$$
		For a dimension vector $d$ of slope $\slope(d) = 1$, i.e.\ $d = (m \rla 0)$, a $\hat{\theta}$-stable representation is nothing but a $m \times r$-matrix $C$ and thus, under the isomorphism $\HH^{1,\sst}(\tildeAOne) \cong \HH(\ptquiv)$, we have an identification
		$$
			\hat{\HH}^{1,\sst}(\tildeAOne,r \rla s) \cong \hat{\HH}(\ptquiv,r).
		$$
		In the same vein, we obtain $\hat{\HH}^{-1,\sst}(\tildeAOne,r \rla s) \cong \hat{\HH}(\ptquiv,s)$. Let $d = (m \rla m)$ be a dimension vector of slope $0$. We have seen that $\smash{R_{m \rla m}^{\sst}} = \Gl_m \times M_m$. A framed representation $(A,B,C,D)$ is therefore $\smash{\hat{\theta}}$-stable if and only if $A$ is invertible and for all subspaces $U,V \sub \C^m$, not both $m$-dimensional, satisfying (1) to (4), we have $\dim U < \dim V$. We see at once that mapping $(A,B,C,D) \mapsto (BA, (C | BD))$ gives an isomorphism $\smash{\hat{M}^{\theta}}(\tildeAOne,r \rla s) \cong \Hilb_{m,r+s}(\loopquiv)$ which is compatible with the isomorphism of moduli stacks constructed on page \pageref{secondEx2}. Therefore, we obtain an isomorphism
		$$
			\hat{\HH}^{0,\sst}(\tildeAOne,r \rla s) \cong \hat{\HH}(\loopquiv,r+s)
		$$
		which is compatible with $\HH^{0,\sst}(\tildeAOne) \cong \HH(\loopquiv)$.
	\end{ex}

	Let $\psi_i^+ \in \HH_{1 \rla 0}$, $\phi_i \in \HH_{1 \rla 1}$ and $\psi_i^- \in \HH_{0 \rla 1}$ be as in Corollary \ref{freeSuper}. We have seen that $\HH(\tildeAOne)$ is the free super-commutative algebra generated by these elements. Combining Corollaries \ref{freeSuper} and \ref{cyclicHilb}, we obtain

	\begin{cor}
		As a graded vector space, $\hat{\HH}(\tildeAOne, r \rla s)$ is isomorphic to the free super-commutative algebra over the vector space spanned by
		$$
			\psi_0^+,\psi_1^+,\ldots,\psi_{r-1}^+, \phi_0,\phi_1,\ldots,\phi_{r+s-1}, \psi_0^-,\psi_1^-,\ldots,\psi_{s-1}^-.
		$$
	\end{cor}

	\begin{proof}
		We show that the kernel, let's call it $\mathcal{I}$, described in Corollary \ref{cyclicHilb} is the ideal of $\HH(\tildeAOne)$ which is generated by $\psi_i^+$ ($i \geq r$), by $\phi_i$ ($i \geq r+s$) and by $\psi_i^-$ ($i \geq s$). It is clear that $\smash{e_{1 \rla 0}^{r \rla s}} \cup \psi_i^+ = \psi_{i+r}^+$ and $\smash{e_{0 \rla 1}^{r \rla s}} \cup \psi_i^- = \psi_{i+s}^-$. Now for the elements $\phi_i(x,y) = x^i \in \HH_{1 \rla 1} = \Q[x,y]$. We have
		$$
			e_{1 \rla 1}^{r \rla s} \cup \phi_i = x^{i+r}y^s.
		$$
		We compute
		$
			\psi_k^+ * \psi_l^- = x^ky^{l+1} - x^{k+1}y^l
		$
		which lies in $\mathcal{I}$ if $k \geq r$ (or $l \geq s$, of course, but we're not using this). Thus
		\begin{align*}
			e_{1 \rla 1}^{r \rla s} \cup \phi_i &= x^{i+r}y^s 
				\equiv x^{i+r+1}y^{s-1} \equiv \ldots 
				\equiv x^{i+r+s}
		\end{align*}
		where $f \equiv g$ means $f-g$ lies in $\mathcal{I}$. Conversely the ideal $\mathcal{I}$ is, by Corollary \ref{cyclicHilb}, generated by elements of $\HH_{1 \rla 0} \oplus \HH_{1 \rla 1} \oplus \HH_{0 \rla 1}$ as $\HH(\tildeAOne)$ is generated by these elements. But for degree reasons we see that $\mathcal{I}$ can't contain more than $(\psi_r^+,\psi_{r+1}^+,\ldots,\phi_{r+s},\phi_{r+s+1},\ldots,\psi_s^-,\psi_{s+1}^-,\ldots)$.
	\end{proof}

	\bibliographystyle{abbrv}
	\bibliography{Literature}
\end{document}